\documentclass{amsart}
\usepackage{hyperref}

\usepackage{graphicx}

\usepackage{graphicx}

\usepackage{amsmath,amssymb,color}

\def\cequiv{\raisebox{-1.5mm}{$\;\stackrel{\raisebox{-3.9mm}{=}}{{\sim}}\;$}}

\input{undertilde}

\def\usigma{\undertilde{\sigma}}
\def\utau{\undertilde{\tau}}
\def\uphi{\undertilde{\varphi}}
\def\upsi{\undertilde{\psi}}
\def\ueta{\undertilde{\eta}}
\def\uzeta{\undertilde{\zeta}}

\def\uf{\undertilde{f}}

\def\uu{\undertilde{u}}
\def\uw{\undertilde{w}}
\def\uv{\undertilde{v}}
\def\ux{\undertilde{x}}

\def\curl{{\rm curl}}
\def\dv{{\rm div}}

\usepackage{tikz}
\usetikzlibrary{snakes}
\usepackage{epic}

\newtheorem{theorem}{Theorem}
\newtheorem{remark}[theorem]{Remark}
\newtheorem{proposition}[theorem]{Proposition}
\newtheorem{lemma}[theorem]{Lemma}

\newtheorem{definition}[theorem]{Definition}
\newtheorem{example}[theorem]{Example}

\begin{document}

\title[Order reduced method for fourth order problems]{Regular decomposition and a framework of order reducd methods for fourth order problems}
\author{
{\sc Shuo Zhang}}
 \address{LSEC, Institute of Computational Mathematics and Scientific/Engineering Computing, Academy of Mathematics and Systems Science, Chinese Academy of Sciences. No 55, Zhongguancun Donglu, Beijing, 100190, China PR. Email: szhang@lsec.cc.ac.cn.
}


\begin{abstract}
This paper is devoted to the construction of order reduced method of fourth order problems. A framework is presented such that a problem on a high-regularity space can be deduced in a constructive way to an equivalent problem on three low-regularity spaces which are connected by a regular decomposition, which is corresponding to a decomposition of the figuration of the regularity of the high order space. The framework is fit for various fourth order problems, and the numerical schemes based on the deduced problems can be of lower complicacy.  Two fourth order problems in three dimensional are discussed under the framework. They are each corresponding to a regular decomposition, and thus are discretised based on the discretised analogues of the regular decompositions constructed; optimal error estimates are given. 
\end{abstract}

\maketitle

\section{Introduction}

Fourth order elliptic problems are frequently encountered in applied sciences. Various model problems fall into this category, including many ones arising from elasticity, electromagnetics, magnetohydrodynamics,  acoustics and etc.. Their discretisations are of theoretical and practical importance. 

The fourth order problems in primal formulations have been widely discussed. Many kinds of conforming and nonconforming finite element methods for second order Sobolev spaces are designed. We refer to, e.g., \cite{Ciarlet.P1978} and \cite{Zenivsek.A1973,Zhang.Shangyou2009,Wang.M;Shi.Z;Xu.J2007NM,Wang.M;Shi.Z;Xu.J2007JCM,Wang.M;Xu.J2007,Tai.X;Winther.R2006,Neilan.M2015,Zheng.B;Hu.Q;Xu.J2011} for a few instances. These elements are designed carefully, and work for certain problems. Particularly, the 2D Morley element for biharmonic equation is generalised by \cite{Wang.M;Xu.J2006} to arbitrary dimensions and then by \cite{Wang.M;Xu.J2013} to arbitrary $(-\Delta)^m$ problem in $n$ dimension with $n\geqslant m$. However, these elements are usually high-degree and designed case by case,  especially in high dimensions. Also, the discretised problems often possess complicated structure and are not easy to implement or to solve. 

An alternative way is to transform the fourth order problems to order reduced formulations. Generally, this is to construct a system on low-regularity spaces by introducing auxiliary variables, and then discretize the generated system with numerical schemes. In this paper, we present a framework of construction of order reduced formulations for fourth order problems. We start with a basic observation that a Sobolev space of second order usually consists of functions that are in some first-order Sobolev space and whose first-order derivatives are belong to some first-order space. A framework is established accordingly with constructive presentation and mathematical analysis. In the framework, an original problem is transformed to a system on three spaces which are connected by a stable decomposition, and the well-define-ness of the three spaces, the well-posedness of the generated system and the equivalence between the generated system and the primal problem are guaranteed. The discretisation of the system needs three finite element spaces of low degree which are connected by some stable decomposition. Once such a stable decomposition can be established, not only a discretisation is given, but also an optimal preconditioner for the generated system can be constructed under the framework of fast auxiliary space preconditioning(FASP)(\cite{Xu.J1996as,Hiptmair.R;Xu.J2007,Xu.J2010icm,Zhang.S;Xu.J2014}) at the computational cost of solving two problems defined on two of the discretised spaces, respectively. 

Principally, the framework could be fit for various fourth order problems where a ``configurated" condition is verified; two specific examples are given in this paper. The first example is the bi-Laplacian equation in three dimensional. The equation is one of fundamental model problems in applied mathematics arising in, e.g., the linear elasticity model in the formulation of Galerkin vector (c.f. \cite{Gurtin.M1973}) and in the transmission eigenvalue problem (c.f. \cite{Colton.D;Monk.P1988,Kirsch.A1986}) in acoustics.  The second example is a fourth order curl equation in three dimensional, which arises from MHD and Maxwell transmission eigenvalue problem(c.f. \cite{Cakoni.F;Haddar.H2007,Zheng.B;Hu.Q;Xu.J2011}). The two problems are each connected to a regular decomposition, and are each transformed to a system on three first-order spaces. By constructing discretised analogues of the regular decompositions, we present discretizations for the generated systems, and prove the optimal convergence rate. 

The framework presented does not try to survey existing order reduced elements. Actually, for various fourth order problems, different order reduced methods have been designed case by case. For example, many of the methods are for the biharmonic equation, including the $u\sim\Delta u$ formulation often known as Ciarlet-Raviart's scheme (\cite{Ciarlet.P;Raviart.P1974}), the $u\sim\nabla^2u$ formulation (related to Hellan-Hermann-Johnson scheme, \cite{Hermann.L1967,Johnson.C1973,Hellan.K1967}), the $u\sim\nabla u\sim\nabla^2 u\sim\dv\nabla^2u$ formulation (Behrens-Guzman's scheme, \cite{Behrens.E;Guzman.J2011}) and the two-dimensional stream Stokes formulation (c.f. Section 5.2 of \cite{GiraultRaviart1986}). These methods are practically useful for discretising biharmonic equation. When taking two dimensional biharmonic equation into consideration, the formulation generated under our framework can be viewed different from but relevant to the stream Stokes formulation which is also used in \cite{Zhang.S;Xi.Y;Ji.X2016} for multilevel scheme for biharmonic eigenvalue problem, and different from others.

There are also numerical schemes designed for fourth order problems which aim to reduce the degree of polynomials used, such as the discontinuous Galerkin(dG) method (c.f., e.g.,\cite{Brenner.S;Monk.P;Sun.J2015,Brenner.S;Sung.L2005,Georgoulis.E;Houston.P2009}) and the weak Galerkin(wG) method (c.f., e.g., \cite{Zhang.R;Zhai.Q2015,Mu.L;Wang.J;Wang.Y;Ye.X2013,Wang.C;Wang.J2014}) for biharmonic equations. These methods still consider the approximation of second order spaces element by element; they use kinds of stabilisations for some flexibility, and the bilinear forms are usually mesh-dependent. We do not discuss them much in the present paper. 

Finally we remark that, as the order reduced formulations presented in this paper are equivalent to the primal formulations, it could be natural to expect equivalence between some discretised problem in order reduced formulation and some discretised problem in primal formulation. Also, their comparison and cooperation can be of interests and be discussed in future.

The remaining of the paper is organised as follows. Some notation is given in the remaining part of this section. In Section \ref{sec:abs}, an order reduction process is presented as a framework of the paper. In Section \ref{sec:rds}, two examples of discretized regular decompositions are given. In Sections \ref{sec:mixbl} and \ref{sec:mixfoc}, the framework is used for three dimensional bi-Laplacian equation and fourth order curl equation by the aid of the discretized regular decompositions. Finally, some concluding remarks are given in Section \ref{sec:con}. 

In this paper, we apply the following notations. Let $\Omega$ be a simply connected polyhedron domain, and $\Gamma=\partial\Omega$ be the boundary of $\Omega$. We use $H^2_0(\Omega)$, $H^1_0(\Omega)$, $H_0(\curl,\Omega)$ and $H_0(\dv,\Omega)$ for certain Sobolev spaces as usual, and specifically, denote  $\displaystyle L^2_0(\Omega):=\{w\in L^2(\Omega):\int_\Omega w dx=0\}$, $\undertilde{H}{}^1_0(\Omega):=(H^1_0(\Omega))^3$, $N_0(\curl,\Omega):=\{\ueta\in H_0(\curl,\Omega), (\ueta,\nabla s)=0\ \forall\,s\in H^1_0(\Omega)\}$, $\mathring{H}_0(\dv,\Omega):=\{\utau\in H_0(\dv,\Omega):\dv\utau=0\}$, and $\undertilde{H}{}^1_0(\curl,\Omega):=\{\uv\in H_0(\curl,\Omega):\nabla\times\uv\in \undertilde{H}{}^1_0(\Omega)\}$.  We use $``\undertilde{~}"$ for vector valued quantities in the present paper. We use $(\cdot,\cdot)$ for $L^2$ inner product and $\langle\cdot,\cdot\rangle$ for the duality between a space and its dual. Without ambiguity, $\langle\cdot,\cdot\rangle$ can occasionally be treated as $L^2$ inner product for certain functions. Finally, $\lesssim$, $\gtrsim$, and $\cequiv$ respectively denote $\leqslant$, $\geqslant$, and $=$ up to a constant. The hidden constants depend on the domain, and, when triangulation is involved, they also depend on the shape-regularity of the triangulation, but they do not depend on $h$ or any other mesh parameter.

\section{Order reduced method for high order problem}
\label{sec:abs}

Let $V$ be a Hilbert function space, and $a(\cdot,\cdot)$ be equivalently an inner product on $V$. We consider the variational problem: given $f\in V'$, find $u\in V$, such that 
\begin{equation}\label{eq:primal}
a(u,v)=\langle f,v\rangle,\quad\forall\,v\in V.
\end{equation}

In many applications, $V$ can be some Sobolev space of second order, and \eqref{eq:primal} is a fourth order problem. In this section, we present an order reduction framework for fourth order problems.

\subsection{Configurable triple}
% R=rough, S=smooth, H=huge
\begin{lemma}
Let $R$, $S$ and $H$ be three Hilbert spaces such that $S\subset H$ continuously, and let  $B$ be a closed operator that maps $R$ into $H$. Define $W:=\{w\in R:Bw\in S\}$. 
\begin{enumerate}
\item $W$ is a Hilbert space with respect to the norm 
\begin{equation}
\|w\|_W:=\|w\|_{R}+\|Bw\|_{S}.
\end{equation}
\item Given $f\in W'$, there are $f_1\in R'$ and $f_2\in S'$, such that $\|f\|_{W'}\geqslant C (\|f_1\|_{R'}+\|f_2\|_{S'})$, and $\langle f,w\rangle=\langle f_1,w\rangle+\langle f_2,Bw\rangle$.
\end{enumerate}
\end{lemma}

\begin{proof}
Obviously, $w\in W$ if and only if $\|w\|_W<\infty$. Now let $\{w_j\}_{1}^\infty$ be a Cauchy sequence with respect to the norm $\|\cdot\|_W$, namely $\|w_i-w_j\|_{R}+\|Bw_i-Bw_j\|_{S}\to 0$ as $i,j\to \infty$, then $\|w_i-w_j\|_{R}\to 0$ and $\|Bw_i-Bw_j\|_{S}\to 0$ as $i,j\to \infty$. By the completeness of $R$ and $S$, there exists a $w_\infty\in R$ and $(Bw)_\infty\in S$, such that $\|w_j-v_\infty\|_{R}+\|Bw_j-(Bv)_\infty\|_{S}\to 0$ as $j\to \infty$. It follows that $\|w_j-w_\infty\|_{R}+\|Bw_j-(Bw)_\infty\|_{H}\to 0$ as $j\to \infty$. As $B$ is a closed operator from $R$ to $H$, $(Bw)_\infty=Bw_\infty$, namely $\|w_j-w_\infty\|_W\to 0$ as $j\to \infty$. 

Let $a_{R/ S}(\cdot,\cdot)$ be the inner product defined on $R/ S$. Then $\frac{1}{2}\|w\|_W^2\leqslant a_R(w,w)+a_S(Bw,Bw)\leqslant \|w\|_W^2$ for $w\in W$. Thus $a(w,v):=a_R(w,v)+a_S(Bw,Bv)$ is equivalently the inner product of $W$. Given $f\in W'$, there exists a $w_f\in W$, such that $\langle f,w\rangle=a(w_f,v)$ for any $v\in W$, and $\|w_f\|_W\leqslant C\|f\|_{W'}$. Obviously $a_R(w_f,\cdot)$ defines and $f_1\in R'$ while $\|f_1\|_{ R'}\leqslant \|w_f\|_{R}$ and $a_S(Bw_f,\cdot)$ defines an $f_2\in S'$ while $\|f_2\|_{S'}\leqslant\|Bw_f\|_{S}$. This completes the proof constructively. 
\end{proof}

\begin{definition}
Given two Hilbert spaces $R$ and $S$ and an operator $B$, if there is a Hilbert space $H$, such that $S\subset H$ continuously, and $B$ is a closed operator that maps $R$ into $H$, then the triple $\{R,S,B\}$ is called a \textbf{configurable} triple with respect to $H$, and $H$ is called the ground space of $\{R,S,B\}$. Given a Hilbert space $W$ equipped with norm $\|\cdot\|_W$, if $W=\{w\in R:Bw\in S\}$, and $\|w\|_W$ is equivalent to $\|w\|_R+\|Bw\|_S$, then $W$ is called to be \textbf{configurated} by the triple $\{R,S,B\}$.
\end{definition}
\begin{remark}
The ``configurable triple" defined here is relevant to the ``compatible" pair defined in \cite{Bergh.J;Lofstrom.J1976}. 
\end{remark}

Denote $BR:=\{Br:r\in R\}.$ As $BR\subset H$ and $S\subset H$, we can define 
$$
Y:=BR+S=\{Br+s:r\in R,\ s\in S\}.
$$
Note that $B$ is closed from $R$ to $H$, and thus $B$ is continuous and $BR$ is closed in $H$. Then, by Lemma 2.3.1 of \cite{Bergh.J;Lofstrom.J1976}, $BR+S$ is complete with respect to the norm
$$
\|y\|_{Y}:=\inf_{r\in R,s\in S, y=Br+s}\|Br\|_{H}+\|s\|_{S}.
$$
By definition, $BR\subset Y$, $S\subset Y$, $\|Br\|_Y\leqslant \|Br\|_{H}$ for $r\in R$ and $\|s\|_{Y}\leqslant \|s\|_{S}$ for $s\in S$. Moreover, since $S\subset H$ continuously, the estimation holds below. 
\begin{lemma}
For any $r\in R$, $\|Br\|_{H}\leqslant C\|Br\|_{Y}$.
\end{lemma}
\begin{proof}
By definition, $\|Br\|_{Y}\geqslant 1/2( \|Br'\|_{H}+\|s'\|_{S})$ for some $r'\in R$ and $s'\in S$ such that $Br=Br'+s'$. Further $2\|Br\|_{Y}\geqslant \|Br'\|_{H}+\|s'\|_{S}\geqslant \|Br'\|_{H}+\|s'\|_{H}\geqslant \|Br\|_{H}$, which completes the proof. 
\end{proof}

\begin{lemma}
$BR$ is closed in $Y$.
\end{lemma}

\begin{example}
It can be easy to find second order Sobolev spaces that are configurated. For example:
\begin{enumerate}
\item $\{H^1_0(\Omega),\undertilde{H}{}^1_0(\Omega),\nabla\}$ is a configurable triple with respect to $L^2(\Omega)$; $H^{2}_0(\Omega)$ is configurated by $\{H^1_0(\Omega),\undertilde{H}{}^1_0(\Omega),\nabla \}$; $H_0(\curl,\Omega)=\nabla H^1_0(\Omega)+\undertilde{H}{}^1_0(\Omega)$;
\item $\{H_0(\curl,\Omega),\undertilde{H}{}^1_0(\Omega),\curl\}$ is a configurable triple with respect to $L^2(\Omega)$; $\undertilde{H}{}^1_0(\curl,\Omega)$ is configurated by $\{H_0(\curl,\Omega),\undertilde{H}{}^1_0(\Omega),\curl\}$; $H_0(\dv,\Omega)=\curl H_0(\curl,\Omega)+\undertilde{H}{}^1_0(\Omega)$.
\end{enumerate}
\end{example}

\subsection{Order reduction of high order problem}
Below we present an order reduction process of \eqref{eq:primal}. 

Let $\{R,S,B\}$ be a configurable triple, and $V$ be configurated by the triple $\{R,S,B\}$. Assume $a(\cdot,\cdot)$ can be represented as
\begin{equation}
a(u,v):=a_R(u,v)+b(u,Bv)+b(v,Bu)+a_S(Bu,Bv),
\end{equation}
where $a_R(\cdot,\cdot)$ $a_S(\cdot,\cdot)$ are two bounded symmetric semidefinite bilinear forms on $R$ and $S$ and $b(\cdot,\cdot)$ is a bounded bilinear form on $R\times S$, and $f(\cdot)$ is represented as
\begin{equation}
f(v)=\langle f_R,v\rangle+\langle f_S,Bv\rangle,\ \ f_R\in R',\ f_S\in S'.
\end{equation}
This way, the problem \eqref{eq:primal} is then to find $u\in V$, such that 
\begin{equation}\label{eq:vpexp}
a_R(u,v)+b(u,Bv)+b(v,Bu)+a_S(Bu,Bv)=\langle f_R,v\rangle+\langle f_S,Bv\rangle,\ \ \forall\,v\in V.
\end{equation}
We assume $a(\cdot,\cdot)$ be coercive on $V$, but $a_{R/S}(\cdot,\cdot)$ is not necessarily an inner product on $R/S$.
Evidently, it is equivalent to find $(u,\phi)\in V\times S$, such that $Bu=\phi$, and 
\begin{equation}
\label{eq:relaxonce}
a_R(u,v)+b(u,\psi)+b(v,\phi)+a_S(\phi,\psi)=\langle f_R,v\rangle+\langle f_S,\psi\rangle,
\\
\forall\,v\in V,\ \mbox{and}\ \psi\in S,\ \mbox{such\ that}\ Bv=\psi.
\end{equation}

The lemma below is evident.
\begin{lemma}\label{lem:fund}
Given $r\in R$ and $s\in S$, the four items below are equivalent: 
\begin{enumerate}
\item  $Br=s$;
\item $r\in V$ and $\langle l_2,Br-s\rangle=0$ for any $l_2\in S'$;
\item $s\in BR$, and $\langle l_1,Br-s\rangle=0$ for any $l_1\in (BR)'$;
\item $\langle l_+,Br-s\rangle=0$ for any $l_+\in Y'$.
\end{enumerate}
\end{lemma}
\begin{remark}
Lemma \ref{lem:fund} connects \eqref{eq:relaxonce} to several equivalent variants of the variational problem. The second item of Lemma \ref{lem:fund} corresponds to the primal formulation of the variation problem, which relies on the configuration of the space $V$; the third item presents that an equivalent problem can be constructed on $BR\cap S$, while it relies on the figuration of $BR\cap S$; the fourth item presents an intermediate mechanism between the former two ones: the regularity is lower than the second item, and the space figuration is looser than the third. 
\end{remark}
Then an equivalent formulation of \eqref{eq:relaxonce} is to find $(u,\phi,g_+)\in R\times S\times Y'$, such that, for any $(v,\psi,l_+)\in R\times S\times Y'$,
\begin{equation}%\label{eq:relaxtwice}
\left\{
\begin{array}{cccl}
\displaystyle a_R(u,v)&+b(v,\phi)&+\langle g_+,Bv\rangle &=\langle f_R,v\rangle,
\\
\displaystyle b(u,\psi)&+a_S(\phi,\psi)&-\langle g_+,\psi\rangle&=\langle f_S,\psi\rangle,
\\
\displaystyle \langle l_+,Bu\rangle&-\langle l_+,\phi\rangle&&=0.
\end{array}
\right.
\end{equation}
Further, let $c(\cdot,\cdot)$ be an inner product defined on $Y$. Given $l_+\in Y'$, there exists a $\zeta_l$, such that $\langle l_+,\eta\rangle=c(\zeta_l,\eta)$ for any $\eta\in Y$, and $\|l_+\|_{Y'}\cequiv \|\zeta_l\|_{Y}$. Thus an equivalent formulation of \eqref{eq:relaxonce} is to find $(u,\phi,\zeta)\in X:=R\times S\times Y$, such that, for $(v,\psi,\eta)\in X$,
\begin{equation}\label{eq:relaxtwice}
\left\{
\begin{array}{cccl}
\displaystyle a_R(u,v)&+b(v,\phi)&+c(Bv,\zeta)&=\langle f_R,v\rangle,
\\
\displaystyle b(u,\psi)&+a_S(\phi,\psi)&-c(\psi,\zeta)&=\langle f_S,\psi\rangle,
\\
\displaystyle c(Bu,\eta)&-c(\phi,\eta)&&=0.
\end{array}
\right.
\end{equation}

A main result of this paper is the theorem below. 
\begin{theorem}\label{thm:abs}
Let $\{R,S,B\}$ be a configurable triple, and $V$ be configurated by the triple $\{R,S,B\}$. Given $f_R\in R'$ and $f_S\in S'$, the problem \eqref{eq:relaxtwice} admits a unique solution $(u,\phi,\zeta)\in X$, and 
$$
\|u\|_{R}+\|\phi\|_{S}+\|\zeta\|_{Y}\cequiv \|f_R\|_{R'}+\|f_S\|_{S'}.
$$
Moreover, $u$ solves the primal problem \eqref{eq:vpexp}. 
%\eqref{eq:primal}. 
\end{theorem}
\begin{proof}
We only have to verify Brezzi's conditions. The coercivity follows from the ellipticity of $a(\cdot,\cdot)$ on V. It remains to verify the inf-sup condition 
\begin{equation}
\inf_{\eta\in Y\setminus\{\mathbf{0}\}}\sup_{(u,\phi)\in R\times S}\frac{c(Bu-\phi,\eta)}{(\|u\|_{R}+\|\phi\|_{S})\|\eta\|_Y}\geqslant C>0.
\end{equation}
for which it suffices to prove, given $\eta\in Y$, there exists $r\in R$ and $s\in S$, such that $Br+s=Y$, and $\|r\|_{R}+\|s\|_{S}\leqslant C\|\eta\|_Y$.

Actually, as $B$ is closed from $R$ to $H$, its kernel space $ker(B):=\{w\in R:Bw=\mathbf{0}\}$ is closed, and $B$ is an isometric operator between $(ker(B))^\perp$, the orthogonal complete of $ker(B)$ in $R$, and $BR$. Now, by definition of $Y$, given $\eta\in Y$, there exists a $p\in BR$ and $s\in S$, such that $\|p\|_{H}+\|s\|_{S}\leqslant C\|\eta\|_Y$, and $p+s=y.$ Further, there exists a unique $r\in (ker(B))^\perp\subset R$, such that $Br=p$. Then $\|r\|_{R}$ is equal to $\|p\|_{H}$ up to a constant and $Br+s=y$. The proof is completed. 
\end{proof}

\subsection{Discretization of \eqref{eq:relaxtwice}}
By Theorem \ref{thm:abs}, \eqref{eq:relaxtwice} presents an equivalent formulation of \eqref{eq:vpexp}.  To construct an order reduced formulation, people only have to seek a triple $\{R,S,B\}$ and to figure out the space $Y=BR+S$. This helps construct low-order numerical schemes, and a discretised analogue of the stable decomposition $Y=BR+S$ plays a key role. 

We suppose subspaces $R_h\subset R$, $S_h\subset S$ and $Y_h\subset Y$ are respectively closed. Moreover, we suppose there is an operator $\Pi_h$ to $Y_h$, such that 
$Y_h=\Pi_h(BR_h+S_h)$. Then we consider the variational problem: find $(u_h,\phi_h,\zeta_h)\in X_h:=R_h\times S_h\times Y_h$, such that, for any $(v_h,\psi_h,\eta_h)\in X_h$, 
\begin{equation}\label{eq:relaxtwicedis}
\left\{
\begin{array}{cccl}
\displaystyle a_R(u_h,v_h)&+b(v_h,\phi_h)&+c(\Pi_hBv_h,\zeta_h)&=\langle f_R,v_h\rangle,
\\
\displaystyle b(u_h,\psi_h)&+a_S(\phi_h,\psi_h)&-c(\Pi_h\psi_h,\zeta_h)&=\langle f_S,\psi_h\rangle,
\\
\displaystyle c(\Pi_hBu_h,\eta_h)&-c(\Pi_h\phi_h,\eta_h)&&=0.
\end{array}
\right.
\end{equation}
Let $\{X_h\}_{h\to 0}$ be a family of discretised spaces. We propose the hypotheses below on the family: there is a constant $C$, such that 
\begin{enumerate}
\item $\|\Pi_hBr_h\|_Y\leqslant C\|r_h\|_R$, $\|\Pi_hs_h\|_Y\leqslant C\|s_h\|_S$, for $r_h\in R_h$ and $s_h\in S_h$;
\item $\|r_h\|_R^2+\|s_h\|_S^2\leqslant C(a_R(r_h,r_h)+b(r_h,s_h)+b(s_h,r_h)+a_S(s_h,s_h))$ for $(r_h,s_h)\in Z_h:=\{(r_h,s_h)\in R_h\times S_h:c(\Pi(Br_h-s_h),\eta_h)=0,\forall\,\eta_h\in Y_h\}$;
\item $\displaystyle\sup_{y_h\in Y_h\setminus\{\mathbf{0}\}}\inf_{r_h\in R_h,s_h\in S_h,\Pi_h(Br_h+s_h)=y_h}\frac{\|r_h\|_R+\|s_h\|_S}{\|y_h\|_Y}\leqslant C$.
\end{enumerate}
\begin{lemma}\label{lem:absconv}
Provided the hypotheses above, the problem \eqref{eq:relaxtwicedis} admits a unique solution $(u_h,\phi_h,\zeta_h)\in X_h$. Let $(u,\phi,\zeta)$ be the solution of \eqref{eq:relaxtwice}. Then
\begin{multline}\label{eq:absconv}
\|u-u_h\|_R+\|\phi-\phi_h\|_S+\|\zeta-\zeta_h\|_Y\leqslant C\Big[\inf_{(v_h,\psi_h,\eta_h)\in X_h}\|(u-v_h,\phi-\psi_h,\zeta-\eta_h)\|_V
\\
+\sup_{v_h\in R_h}\frac{c(Bv_h-\Pi_hBv_h,\zeta)}{\|v_h\|_R}+\sup_{\psi_h\in S_h}\frac{c(\psi_h-\Pi_h\psi_h,\zeta)}{\|\psi_h\|_S}+\sup_{\eta_h\in Y_h}\frac{c((Bu-\phi)-\Pi_h(Bu-\phi),\eta_h)}{\|\eta_h\|_Y}\Big].
\end{multline}
\end{lemma}
\begin{proof}
The hypotheses verify Brezzi's conditions, and the well-posedness of the discretised problem follows. \eqref{eq:absconv} is a Strang-type estimate, and we refer to, e.g., Proposition 5.5.6 of \cite{Boffi.D;Brezzi.F;Fortin.M2013}. The proof is completed.
\end{proof}

\begin{remark}
Provided the hypotheses, roughly speaking, \eqref{eq:relaxtwicedis} induces a stable bilinear form on $X_h$, thus an optimal preconditioner can be constructed for \eqref{eq:relaxtwicedis} by the aid of auxiliary problems constructed on $R_h$, $S_h$ and $Y_h$(\cite{Ruesten.T;Winther.R1992}). Moreover, by the stable decomposition (the third of the hypotheses), a fast auxiliary space preconditioner can be constructed for \eqref{eq:relaxtwicedis} with auxiliary problems constructed on $R_h$ and $S_h$(\cite{Xu.J1996as,Hiptmair.R;Xu.J2007,Xu.J2010icm,Zhang.S;Xu.J2014}). 
\end{remark}

\section{Continuous and discretized regular decompositions}
\label{sec:rds}

\subsection{Two stable regular decompositions}

\begin{lemma}\label{lem:capreg}
(\cite{GiraultRaviart1986})For polyhedron domain $\Omega$, $H_0(\curl,\Omega)\cap H_0(\dv,\Omega)=\undertilde{H}{}^1_0(\Omega).$ Moreover, if $\Omega$ is convex, $H_0(\curl,\Omega)\cap H(\dv,\Omega)\subset \undertilde{H}^1(\Omega),$ and $H(\curl,\Omega)\cap H_0(\dv,\Omega)\subset \undertilde{H}^1(\Omega).$
\end{lemma}

\begin{lemma}\label{lem:curlsur} (\cite{Birman.M;Solomyak.M1987,Dhia.A;Hazard.C;Lohrengel.S1999,Pasciak.J;Zhao.J2002,Hiptmair.R;Xu.J2007})
Given $\ueta\in H_0(\curl,\Omega)$, there exists a $\uphi\in\undertilde{H}{}^1_0(\Omega)$, such that $\curl\uphi=\curl\ueta$, and $\|\uphi\|_{1,\Omega}\leqslant C\|\ueta\|_{\curl,\Omega}$, with $C$ a generic positive constant uniform in $H_0(\curl,\Omega)$.
\end{lemma}
\begin{lemma}\label{lem:divsur}
(\cite{GiraultRaviart1986}) Given $\utau\in H_0(\dv,\Omega)$, there exists a $\uphi\in\undertilde{H}{}^1_0(\Omega)$, such that $\dv\uphi=\dv\utau$, and $\|\uphi\|_{1,\Omega}\leqslant C\|\utau\|_{\dv,\Omega}$, with $C$ a generic positive constant uniform in $H_0(\dv,\Omega)$.
\end{lemma}
These two well-known regular decompositions can be derived from Lemmas \ref{lem:curlsur} and \ref{lem:divsur}, respectively. 
\begin{lemma}
A decomposition of $H_0(\curl,\Omega)$ holds that
\begin{equation}\label{eq:rdhcurl}
H_0(\curl,\Omega)=\nabla H^1_0(\Omega)+\undertilde{H}{}^1_0(\Omega).
\end{equation}
Namely, given $\ueta\in H_0(\curl,\Omega)$, there exists a $w\in H^1_0(\Omega)$ and $\uphi\in\undertilde{H}{}^1_0(\Omega)$, such that 
$\ueta=\nabla w+\uphi$ and $\|\ueta\|_{\curl,\Omega}\geqslant C(\|w\|_{1,\Omega}+\|\uphi\|_{1,\Omega})$.

\end{lemma}
\begin{lemma}
A stable decomposition of $H_0(\dv,\Omega)$ holds that 
\begin{equation}\label{eq:rdhdiv}
H_0(\dv,\Omega)=\curl H_0(\curl,\Omega)+\undertilde{H}{}^1_0(\Omega).
\end{equation} 
Namely given $\utau\in H_0(\dv,\Omega)$, there exists a $\uzeta\in H_0(\curl,\Omega)$ and $\uphi\in\undertilde{H}{}^1_0(\Omega)$, such that 
$\utau=\curl\uzeta +\uphi$ and $\|\utau\|_{\dv,\Omega}\geqslant C(\|\uzeta\|_{\curl,\Omega}+\|\uphi\|_{1,\Omega})$.
\end{lemma}

%In this section, we discuss their discrete analogues. 

%
%
\subsection{Subdivisions and finite element spaces}
Let $\mathcal{G}_h$ be a shape-regular tetrahedron subdivision of $\Omega$, such that $\bar\Omega=\cup_{K\in\mathcal{G}_h}\bar K$. Denote by $\mathcal{F}_h$, $\mathcal{F}_h^i$, $\mathcal{E}_h$, $\mathcal{E}_h^i$, $\mathcal{X}_h$ and $\mathcal{X}_h^i$ the set of faces, interior faces, edges, interior edges, vertices and interior vertices, respectively. For any edge $e\in\mathcal{E}_h$, denote by $\undertilde{t}{}_e$ the unit tangential vector along $e$; for any face $f\in\mathcal{F}_h$, denote by $\undertilde{n}{}_f$ the unit normal vector of $f$. Denote by $P_k(\mathcal{G}_h)$ the space of piecewise $k$-th degree polynomials on $\mathcal{G}_h$.
 
For $K$ a tetrahedron, as usual, we use $P_k(K)$ for the set of polynomials on $K$ of degrees not higher than $k$, and $P_k(F)$ for the set of polynomials of degrees not higher than $k$ on a face $F$ of $K$. Denote by $a_i$ and $F_i$ vertices and opposite faces of $K$, $i=1,2,3,4$.  The barycentre coordinates are denoted as usual by $\lambda_i$, $i=1,2,3,4$. 
Denote $q_0=\lambda_1\lambda_2\lambda_3\lambda_4$, and $q_i=\lambda_j\lambda_k\lambda_l$ with $\{j,k,l\}=\{1,2,3,4\}\setminus\{i\}$, $i=1,2,3,4$. Besides, define shape function spaces as
\begin{itemize}
\item $P^e(K):={\rm span}\{\lambda_i\lambda_j,1\leqslant i\neq j\leqslant 4\}$;
\item $P^f(K):={\rm span}\{q_i,1\leqslant i\leqslant 4\}$;
\item $\mathbb{E}(K):=\{\uu+\uv\times\ux:\uu,\uv\in \mathbf{R}^3\}$;
\item $\mathbb{F}(K):=\{\uu+v\ux:\uu\in\mathbb{R}^3,v\in\mathbb{R}\}$.
\end{itemize}
Then define finite element functions/spaces as
\begin{itemize}
\item $\mathsf{L}_h:=\{w\in H^1(\Omega):w|_K\in P_1(K),\ \forall\,K\in \mathcal{G}_h\}$, and $\mathsf{L}_{h0}:=\mathsf{L}_h\cap H^1_0(\Omega)$; $\undertilde{\mathsf{L}}{}_{h(0)}:=(\mathsf{L}_{h(0)})^3$;

\item $\mathsf{L}_h^{e}:=\{w\in H^1(\Omega):w|_K\in P^{e}(K),\ \forall\,K\in \mathcal{G}_h\}$, and $\mathsf{L}_{h0}^{e}:=\mathsf{L}_h^{e}\cap H^1_0(\Omega)$;

\item $\mathsf{L}_h^{f}:=\{w\in H^1(\Omega):w|_K\in P^{f}(K),\ \forall\,K\in \mathcal{G}_h\}$, and $\mathsf{L}_{h0}^{f}:=\mathsf{L}_h^{f}\cap H^1_0(\Omega)$;

\item $b_e\in \mathsf{L}^{e}_h$ such that $b_{e}=0$ on $e'\in\mathcal{E}_h\setminus \{e\}$; $b_f\in\mathsf{L}^{f}_h$ such that $b_{f}=0$ on $f'\in\mathcal{F}_h\setminus\{f\}$;

\item $\undertilde{\mathsf{L}}{}_h^e:={\rm span}\{b_e\undertilde{t}{}_e\}_{e\in \mathcal{E}_h}$, and $\undertilde{\mathsf{L}}{}_{h0}^e=\undertilde{\mathsf{L}}{}_h^e\cap \undertilde{H}{}^1_0(\Omega)$;

\item $\undertilde{\mathsf{L}}{}_{h}^{f}:={\rm span}\{b_f\undertilde{n}{}_f\}_{f\in \mathcal{F}_h}$, and $\undertilde{\mathsf{L}}{}_{h0}^{f}:=\undertilde{\mathsf{L}}{}_{h}^{f}\cap \undertilde{H}{}^1_0(\Omega)$;

\item $\undertilde{\mathsf L}{}_{h0}^{+e}:=\undertilde{\mathsf L}{}_{h0}+\undertilde{\mathsf L}{}_{h0}^e$; $\undertilde{\mathsf L}{}_{h0}^{+f}:=\undertilde{\mathsf L}{}_{h0}+\undertilde{\mathsf L}{}_{h0}^f$;

\item $\mathbb{N}_h:=\{\uw\in H(\curl,\Omega):\uw|_K\in \mathbb{E}(K),\ \forall\,K\in \mathcal{G}_h\}$, and $\mathbb{N}_{h0}:=\mathbb{N}_h\cap H_0(\curl,\Omega)$;

\item $\mathbb{RT}_h:=\{\uw\in H(\dv,\Omega):\uw|_K\in \mathbb{F}(K),\ \forall\,K\in \mathcal{G}_h\}$, and $\mathbb{RT}_{h0}:=\mathbb{RT}_{h}\cap H_0(\dv,\Omega)$.
\end{itemize}

Finally, use $\mathcal{L}^0_{h}$ for the space of piecewise constant, and $\mathcal{L}^0_{h0}:=\mathcal{L}^0_{h}\cap L^2_0(\Omega)$. 

\begin{lemma}\label{lem:sdH1}
There exists a constant $C$, such that, for any $w_1\in\mathsf{L}_{h0}$ and $w_2\in\mathsf{L}^e_{h0}$(or $w_2\in\mathsf{L}^f_{h0}$),
\begin{equation}
\|w_1+w_2\|_{1,\Omega}\cequiv C\left[\|w_1\|_{1,\Omega}+(\sum_{K\in\mathcal{G}_h}h_K^{-2}\|w_2\|_{0,K}^2)^{1/2}\right].
\end{equation}
\end{lemma}
\begin{proof}
Given $K\in \mathcal{G}_h$, direct calculation leads to that 
$$
\|\nabla(p+q)\|_{0,K}^2\geqslant C_K(\|\nabla p\|_{0,K}^2+h_k^{-2}\| q\|_{0,K}^{2}),
$$
for any $p\in P_1(K)$ and $q\in P^e(K)$ or $q\in P^f(K)$ with $C_K$ a constant depending on the regularity of $K$. Make a summation on all cells $K$, and we can obtain 
$$
\|\nabla (w_1+w_2)\|_{0,\Omega}\geqslant C(\|\nabla w_1\|_{0,\Omega}+(\sum_{K\in\mathcal{G}_h}h_K^{-2}\|w_2\|_{0,K}^2)^{1/2}).
$$
The other direction holds by the triangle inequality and the inverse estimate. The proof is completed. 
\end{proof}

Let $\Pi_h^\mathbb{N}$ be the nodal interpolation to $\mathbb{N}_h$ defined such that, for $\ueta$ that makes the operation doable,
$$
\int_e\Pi_h^\mathbb{N}\ueta\cdot\undertilde{t}{}_e=\int_e\ueta\cdot\undertilde{t}{}_e,\ \ \forall\,e\in \mathcal{E}_h,
$$
and let $\Pi_h^\mathbb{RT}$ be the nodal interpolation to $\mathbb{RT}_h$ defined such that, for $\utau$ that makes the operation doable,
$$
\int_f\Pi_h^\mathbb{RT}\utau\cdot\undertilde{n}{}_f=\int_f\utau\cdot\undertilde{n}{}_f,\ \ \forall\,f\in \mathcal{F}_h.
$$
\begin{lemma}\label{lem:HXlemma}(Lemma 5.1 of \cite{Hiptmair.R;Xu.J2007})
\begin{enumerate}
\item For any $\ueta{}_h\in \mathbb{N}_{h0}$, there are $\upsi{}_h\in \undertilde{\mathsf L}{}_{h0}$, $p_h\in \mathsf{L}_{h0}$, and $\tilde{\ueta}{}_h\in\mathbb{N}_{h0}$, such that 
$$
\ueta{}_h=\tilde{\ueta}{}_h+\Pi_h^\mathbb{N}\upsi{}_h+\nabla p_h,
$$
and
$$
\|h^{-1}\tilde{\ueta}{}_h\|_{0,\Omega}^2+\|\upsi{}_h\|_{1,\Omega}^2+\|p_h\|_{1,\Omega}^2\leqslant \|\ueta{}_h\|_{\curl,\Omega}^2.
$$

\item For any $\utau{}_h\in \mathbb{RT}_{h0}$, there are $\upsi{}_h\in \undertilde{\mathsf L}{}_{h0}$, $\ueta{}_h\in \mathbb{N}_{h0}$, and $\tilde{\utau}{}_h\in\mathbb{RT}_{h0}$, such that 
$$
\utau{}_h=\tilde{\utau}{}_h+\Pi_h^\mathbb{RT}\upsi{}_h+\curl \ueta{}_h,
$$
and
$$
\|h^{-1}\tilde{\utau}{}_h\|_{0,\Omega}^2+\|\upsi{}_h\|_{1,\Omega}^2+\|\ueta{}_h\|_{\curl,\Omega}^2\leqslant C \|\utau{}_h\|_{\dv,\Omega}^2.
$$
\end{enumerate}
\end{lemma}

\begin{lemma}\label{lem:intcurlpolyn}
For any integer $s>0$, there is a $C=C(s)>0$, such that 
$$
\|\uphi-\Pi_h^\mathbb{N}\uphi\|_{0,\Omega}\leqslant Ch|\uphi|_{1,\Omega},\ \ \forall\,\uphi\in(P_s(\mathcal{G}_h))^3\cap \undertilde{H}{}^1(\Omega).
$$

\end{lemma}
\begin{proof}
Let $K$ be a tetrahedron of $\mathcal{G}_h$. By scaling argument (c.f. Page 257 of 
\cite{Boffi.D;Conforti.M;Gastaldi.L2006}), given $p>2$, there exists a constant $C_p>0$, such that if $\ueta\in \undertilde{H}^1(K)$ and $\curl\upsi\in \undertilde{L}^p(K)$, then 
$$
\|\upsi-\Pi_h^\mathbb{N}\upsi\|_{0,K}\leqslant C_p(h_K|\upsi|_{1,K}+h_K^{5/2-3/p}\|\curl\upsi\|_{L^p(K)}),\ \ \forall\,K\in \mathcal{G}_h.
$$

Therefore, by inverse inequality, given $\uphi\in (P_s(K))^3$,
\begin{multline}
\|\uphi-\Pi_h^\mathbb{N}\uphi\|_{0,K}\leqslant C_p(h_K|\uphi|_{1,K}+h_K^{5/2-3/p}\|\curl\uphi\|_{L^p(K)})
\\
\leqslant C_p(h_K|\uphi|_{1,K}+C_sh_K^{5/2-3/p}h_K^{3/p-3/2}\|\curl\uphi\|_{L^2(K)})\leqslant C(s)h_K|\uphi|_{1,K}.
\end{multline}
The proof is completed. 
\end{proof}

\subsection{A discretised analogue of the decomposition \eqref{eq:rdhcurl}}

We are going to construct the result below.

\begin{theorem}\label{thm:dsdhcurl}
The decomposition below is stable:
\begin{equation}
\mathbb{N}_{h0}=\nabla \mathsf{L}_{h0}+\Pi_h^\mathbb{N} \undertilde{\mathsf{L}}{}^{+e}_{h0}.
\end{equation}
Precisely, given $\ueta{}_h\in \mathbb{N}_{h0}$, there exist a $w_h\in \mathsf{L}_{h0}$ and a $\uphi{}_h\in\undertilde{\mathsf{L}}{}^{+e}_{h0}$, such that 
\begin{equation}
\|w_h\|_{1,\Omega}+\|\uphi{}_h\|_{1,\Omega}\leqslant C\|\ueta{}_h\|_{\curl,\Omega}\quad\mbox{and}\ \ \ \ueta{}_h=\nabla w_h+\Pi_h^\mathbb{N}\uphi{}_h.
\end{equation}
\end{theorem}

\begin{proof}
Given $\ueta{}_h\in \mathbb{N}_{h0}$, by Lemma \ref{lem:HXlemma}, there exist $\tilde{\ueta}{}_h\in\mathbb{N}_{h0}$, $\upsi{}_h\in \undertilde{\mathsf{L}}{}_{h0}$ and $w_h\in \mathsf{L}{}_{h0}$, such that 
$$
h^{-1}\|\tilde{\ueta}{}_h\|_{0,\Omega}+\|\upsi{}_h\|_{1,\Omega}+\|w_h\|_{1,\Omega}\leqslant C\|\ueta{}_h\|_{\curl,\Omega},
$$
and
$$
\ueta{}_h=\tilde{\ueta}{}_h+\Pi_h^\mathbb{N}\upsi{}_h+\nabla w_h.
$$
Let $\uphi{}_h^\eta\in \undertilde{\mathsf{L}}{}_{h}^e$ be such that $\int_e\uphi{}_h^\eta\cdot\undertilde{t}{}_e=\int_e\tilde{\ueta}{}_h\cdot\undertilde{t}{}_e$ for any $e\in \mathcal{E}_h$. Then $\tilde{\ueta}{}_h=\Pi_h^\mathbb{N}\uphi{}_h^\eta$, and $\|\tilde{\ueta}{}_h\|_{0,\Omega}\cequiv \|\uphi{}_h^\eta\|_{0,\Omega}$. This way $\ueta{}_h=\Pi_h^\mathbb{N}(\upsi{}_h+\uphi{}_h^\eta)+\nabla w_h$, and $h^{-1}\|\uphi{}_h^\eta\|_{0,\Omega}+\|\upsi{}_h\|_{1,\Omega}+\|w_h\|_{1,\Omega}\leqslant C\|\ueta{}_h\|_{\curl,\Omega}.$ Setting $\uphi{}_h=\upsi{}_h+\uphi{}_h^\eta\in \undertilde{\mathsf L}{}_{h0}^{+e}$, we have by Lemma \ref{lem:sdH1} that
$$
\ueta{}_h=\Pi_h^\mathbb{N}\uphi{}_h+\nabla w_h,
$$
and
$$
\|\uphi{}_h\|_{1,\Omega}+\|w_h\|_{1,\Omega}\leqslant C\|\ueta{}_h\|_{\curl,\Omega}.
$$
This completes the proof. 
\end{proof}

%\subsubsection{A stable $``P_2-P_0"$ pair in 3D}

Note that $\mathring{H}_0(\dv,\Omega)=\curl H_0(\curl,\Omega)=\curl\undertilde{H}{}^1_{0}(\Omega)$, and the inf-sup condition holds below:
\begin{equation}\label{eq:infsuphhdiv}
\inf_{\utau\in\mathring{H}_0(\dv,\Omega)}\sup_{\uphi\in\undertilde{H}{}^1_0(\Omega)}\frac{(\curl\uphi,\utau)}{\|\uphi\|_{1,\Omega}\|\utau\|_{\dv,\Omega}}\geqslant C>0.
\end{equation}

Denote by $\mathring{\mathbb{RT}}_{h0}:=\{\utau\in \mathring{H}_0(\dv,\Omega):\utau|_K\in\mathbb{R}^3\}$; then $\mathring{\mathbb{RT}}_{h0}=\curl \mathbb{N}_{h0}$. A discretized analogue of \eqref{eq:infsuphhdiv} can be constructed. 
\begin{proposition}\label{cor:infsupp2p0}
\begin{equation}
\inf_{\utau{}_h\in \mathring{\mathbb{RT}}_{h0}\setminus\{\mathbf{0}\}}\sup_{\uphi{}_h\in\undertilde{\mathsf{L}}{}_{h0}^{+e}}\frac{(\curl\,\Pi_h^\mathbb{N}\uphi{}_h,\utau{}_h)}{\|\uphi{}_h\|_{1,\Omega}\|\utau{}_h\|_{\dv,\Omega}}\geqslant C>0.
\end{equation}
\end{proposition}
\begin{proof}
Given $\utau{}_h\in\mathring{\mathbb{RT}}_{h0}\setminus\{\mathbf{0}\}$, there exists a $\ueta{}_h\in\mathbb{N}_{h0}$, such that $\utau{}_h=\curl\ueta{}_h$. Then set $\uphi{}_h\in \undertilde{\mathsf{L}}{}_{h0}^{+e}$, such that $\ueta{}_h=\Pi_h^\mathbb{N}\uphi{}_h+\nabla w_h$ and $\|\ueta{}_h\|_{\curl,\Omega}\geqslant C(\|\uphi{}_h\|_{1,\Omega}+\|w_h\|_{1,h})$ with some $w_h\in \mathsf{L}_{h0}$, and we have 
$$
\frac{(\curl\Pi_h^\mathbb{N}\uphi{}_h,\utau{}_h)}{\|\uphi{}_h\|_{1,\Omega}\|\utau{}_h\|_{\dv,\Omega}}\geqslant \frac{(\curl\Pi_h^\mathbb{N}\uphi{}_h,\curl\ueta{}_h)}{\|\uphi{}_h\|_{1,\Omega}\|\ueta{}_h\|_{\curl,\Omega}}\geqslant C>0.
$$
This completes the proof.
\end{proof}

\subsection{A discretised analogue of the decomposition \eqref{eq:rdhdiv}}

Similarly to Theorem \ref{thm:dsdhcurl}, we can construct the result below.

\begin{theorem}\label{thm:dsdhdiv}
The decomposition below is stable:
\begin{equation}
\mathbb{RT}_{h0}=\curl\mathbb{N}_{h0}+\Pi_h^\mathbb{RT} \undertilde{\mathsf{L}}{}^{+f}_{h0}.
\end{equation}
Precisely, given $\utau{}_h\in \mathbb{RT}_{h0}$, there exists a $\uzeta{}_h\in \mathbb{N}_{h0}$ and a $\uphi{}_h\in\undertilde{\mathsf{L}}{}^{+f}_{h0}$, such that 
\begin{equation}
\|\uzeta{}_h\|_{\curl,\Omega}+\|\uphi{}_h\|_{1,\Omega}\leqslant C\|\utau{}_h\|_{\dv,\Omega}\quad\mbox{and}\ \ \utau{}_h=\curl\uzeta{}_h+\Pi_h^\mathbb{RT}\uphi{}_h.
\end{equation}
\end{theorem}
Moreover, the inf-sup condition is well known
$$
\inf_{q\in L^2_0(\Omega)}\sup_{\uphi\in\undertilde{H}{}^1_0(\Omega)}\frac{(\dv\uphi,q)}{\|\uphi\|_{1,\Omega}\|q\|_{0,\Omega}}\geqslant C>0.
$$
Also, similar to Proposition \ref{cor:infsupp2p0}, a discretised inf-sup condition can be proved below. %\mathcal{L}^0_{h0}
\begin{proposition}\label{cor:infsupp3p0}
\begin{equation}\label{eq:infsupp3p0}
\inf_{q_h\in \mathcal{L}^0_{h0}}\sup_{\uphi{}_h\in\mathsf{L}_{h0}^{+f}}\frac{(\dv\Pi_h^\mathbb{RT}\uphi{}_h,q_h)}{\|\uphi{}_h\|_{1,\Omega}\|q_h\|_{0,\Omega}}\geqslant C>0.
\end{equation}
\end{proposition}

\begin{remark}
In \eqref{eq:infsupp3p0}, by the commutative digram, the interpolation operator $\Pi_h^\mathbb{RT}$ can be removed. Actually, by noting the commutative diagram and the roles of the interpolation operators, we can view both the $\undertilde{L}{}_{h0}^{+e}-\mathbb{RT}_{h0}$ pair constructed in Proposition \ref{cor:infsupp2p0} and $\mathsf{L}_{h0}^{+f}-\mathcal{L}^0_{h0}$ pair constructed in Proposition \ref{cor:infsupp3p0} as 3D analogues of the 2D Bernardi-Raugel pair(\cite{Bernardi.C;Raugel.G1985}) for Stokes problem.
\end{remark}

\section{A mixed element method of the 3D bi-Laplacian equation}
\label{sec:mixbl}

In this section, we study the finite element method for the three dimensional bi-Laplacian equation
\begin{equation}\label{eq:modelbl}
(-\Delta)\left(-\alpha(x)\Delta u\right)=f,
\end{equation} 
with the homogeneous boundary data $u=0,\ \nabla u=\undertilde{0}$ and with $0<\alpha_s<\alpha(x)<\alpha_b$.  

\subsection{Order reduced formulation}

The primal variational formulation is, given $f\in H^{-2}(\Omega)$, to find $u\in H^2_0(\Omega)$, such that 
\begin{equation}%\label{eq:primalbl}
(\alpha\Delta u,\Delta v)=\langle f,v\rangle,\quad\forall\,v\in H^2_0(\Omega).
\end{equation}

Note that $H^{2}_0(\Omega)$ is configurated by $\{H^1_0(\Omega),\undertilde{H}{}^1_0(\Omega),\nabla \}$: $H^2_0(\Omega)=\{w\in H^1_0(\Omega):\nabla u\in\undertilde{H}{}^1_0(\Omega)\}$. We rewrite the model problem as, given $f_1\in H^{-1}(\Omega)$ and $\uf{}_2\in \undertilde{H}{}^{-1}(\Omega)$, to find $u\in H^2_0(\Omega)$, such that  
\begin{equation}\label{eq:primalbl}
(\alpha\Delta u,\Delta v)+(\curl\nabla u,\curl\nabla v)=\langle f_1,v\rangle+\langle\uf{}_2,\nabla v\rangle,\quad\forall\,v\in H^2_0(\Omega).
\end{equation}
Note that we add a mute term $(\curl\nabla u,\curl\nabla v)$ here without any difference. 

By Theorem \ref{thm:abs}, a mixed formulation of \eqref{eq:primalbl} is then to find $(u,\uphi,\uzeta)\in M:=H^1_0(\Omega)\times\undertilde{H}{}^1_0(\Omega)\times H_0(\curl,\Omega)$, such that, for any $(v,\upsi,\ueta)\in M$, 
\begin{equation}\label{eq:mixbl}
\left\{
\begin{array}{ccccl}
&&(-\nabla v,\uzeta)&=&\langle f_1,v\rangle,
\\
&(\alpha\dv\uphi,\dv\upsi)+(\curl\uphi,\curl\upsi)&(\upsi,\uzeta)+(\curl\upsi,\curl\uzeta)&=&\langle\uf{}_2,\upsi\rangle,
\\
(-\nabla u,\ueta)&(\uphi,\ueta)+(\curl\uphi,\curl\ueta)&&=&0.
\end{array}
\right.
\end{equation}

\begin{theorem}
Given $f_1\in H^{-1}(\Omega)$ and $\uf{}_2\in \undertilde{H}{}^{-1}(\Omega)$, the equation \eqref{eq:mixbl} admits a unique solution $(u,\uphi,\uzeta)\in M$, and 
$$
\|u\|_{1,\Omega}+\|\uphi\|_{1,\Omega}+\|\uzeta\|_{\curl,\Omega}\cequiv \|f_1\|_{-1,\Omega}+\|\uf{}_2\|_{-1,\Omega}.
$$ 
Moreover, $u$ solves the primal problem \eqref{eq:primalbl}.
\end{theorem}

\begin{theorem}
Let $\Omega$ be convex and $\alpha$ be smooth. Then,
\begin{enumerate}
\item if $\uf{}_2=\undertilde{0}$, then $\|u\|_{3,\Omega}+\|\uphi\|_{2,\Omega}\cequiv \|f_1\|_{-1,\Omega}$;
\item if $f_1\in L^2(\Omega)$ and $\uf{}_2=\undertilde{0}$, then $\|\uzeta\|_{1,\Omega}+\|\curl\uzeta\|_{1,\Omega}\cequiv\|f_1\|_{0,\Omega}$.
\end{enumerate}
\end{theorem}
\begin{proof}
The first item follows from the regularity theory of fourth order elliptic equation and the equivalence between \eqref{eq:mixbl} and \eqref{eq:primalbl}. Now we turn to the second item. Decompose $\uzeta=\nabla w_\zeta+\uzeta{}_1$ with $w_\zeta\in H^1_0(\Omega)$ and $\uzeta{}_1\in N_0(\curl,\Omega)$; the decomposition is unique. Then $(\nabla w_\zeta,\nabla v)=(f,v)$ for any $v\in H^1_0(\Omega)$, thus $w_\zeta\in H^1_0(\Omega)\cap H^2(\Omega)$ and $\|w_\zeta\|_{2,\Omega}\cequiv \|f\|_{0,\Omega}$. By Lemma \ref{lem:capreg}, $\uzeta{}_2\in N_0(\curl,\Omega)\subset \undertilde{H}{}^1(\Omega)$, and thus $\uzeta\in \undertilde{H}^1(\Omega)$ and $\|\uzeta\|_{1,\Omega}\cequiv \|f\|_{0,\Omega}$. For any $\upsi\in \undertilde{H}{}^1_0(\Omega)$, $(\curl\uzeta,\curl\upsi)=-(\uzeta,\upsi)-(\alpha\dv\uphi,\dv\upsi)=(\nabla\alpha\dv\uphi-\uzeta,\upsi)$, which leads to that $\curl\curl\uzeta=\nabla\alpha\dv\uphi-\uzeta\in\undertilde{L}{}^2(\Omega)$ and $\curl\uzeta\in H(\curl,\Omega)$. As $\curl\uzeta\in H_0(\dv,\Omega)$, we obtain $\curl\uzeta\in \undertilde{H}^1(\Omega)$, and $\|\curl\uzeta\|_{1,\Omega}\leqslant C\|f\|_{0,\Omega}$. This completes the proof. 
\end{proof}

\subsection{Discretization}

Define $M_h:=\mathsf{L}_{h0}\times\undertilde{\mathsf L}{}^{+e}_{h0}\times\mathbb{N}_{h0}$. A discretization scheme is to find $(u_h,\uphi{}_h,\uzeta{}_h)\in M_h$, such that, for $(v_h,\upsi{}_h,\ueta{}_h)\in M_h$,
\begin{equation}\label{eq:mixdisbl}
\left\{
\begin{array}{ccccl}

& &(-\nabla v_h,\uzeta{}_h) &=&\langle f_1,v_h\rangle 
\\

&(\alpha\dv\uphi{}_h,\dv\upsi{}_h)+(\curl\uphi{}_h,\curl\upsi{}_h)&(\Pi_h^\mathbb{N}\,\upsi{}_h,\uzeta{}_h)+(\curl \,\Pi_h^\mathbb{N}\,\upsi{}_h,\curl\uzeta{}_h) &=&\langle\uf{}_2,\upsi{}_h\rangle
\\
(-\nabla u_h,\ueta{}_h)&((\Pi_h^\mathbb{N}\,\uphi{}_h),\ueta{}_h)+(\curl \,\Pi_h^\mathbb{N}\,\uphi{}_h,\curl\ueta{}_h)&&=&0.
\end{array}
\right.
\end{equation}

The error estimate for the interpolation below is technically important.

\begin{theorem}
Given $f_1\in H^{-1}(\Omega)$ and $\uf{}_2\in\undertilde{H}^{-1}(\Omega)$, the problem \eqref{eq:mixdisbl} admits a unique solution $(u_h,\uphi{}_h,\uzeta{}_h)\in M_h$, and 
$$
\|u_h\|_{1,\Omega}+\|\uphi{}_h\|_{1,\Omega}+\|\uzeta{}_h\|_{\curl,\Omega}\cequiv \|f_1\|_{-1,h}+\|\uf{}_2\|_{-1,h},
$$
where
$$
\|f_1\|_{-1,h}:=\sup_{v_h\in \mathsf{L}{}_{h0}\setminus\{\mathbf{0}\}}\frac{(f,v_h)}{\|v_h\|_{1,\Omega}}, \ \ \mbox{and}\ \  \|\uf{}_2\|_{-1,h}:=\sup_{\upsi{}_h\in\undertilde{\mathsf L}{}_{h0}^{+e}\setminus\{\mathbf{0}\}}\frac{\langle \uf{}_2,\upsi{}_h\rangle }{\|\upsi{}_h\|_{1,\Omega}}.
$$

\end{theorem}
\begin{proof}
We only have to verify the hypotheses of Lemma \ref{lem:absconv}. By Lemma \ref{lem:intcurlpolyn}, $\|\Pi_h^\mathbb{N}\upsi{}_h\|_{0,\Omega}\leqslant C\|\upsi{}_h\|_{0,\Omega}$ and thus $\|\Pi_h^\mathbb{N}\upsi{}_h\|_{\curl,\Omega}\leqslant C\|\upsi{}_h\|_{1,\Omega}$ for $\upsi{}_h\in\undertilde{\mathsf L}{}_{h0}^{+e}$. This verifies the boundedness of all the bilinear forms involved. Moreover, $(\alpha\dv\uphi{}_h,\dv\uphi{}_h)+(\curl\uphi{}_h,\curl\uphi{}_h)\geqslant C(\|\uphi{}_h\|_{1,\Omega}^2+\|u_h\|_{1,\Omega}^2)$ for $(u_h,\uphi{}_h)\in Z_h:=\{(u_h,\uphi{}_h)\in \mathsf{L}_{h0}\times\undertilde{\mathsf L}{}_{h0}^{+e}:((\Pi_h^\mathbb{N}\,\uphi{}_h),\ueta{}_h)-(\nabla u_h,\ueta{}_h)+(\curl \,\Pi_h^\mathbb{N}\,\uphi{}_h,\curl\ueta{}_h)=0,\forall\,\ueta{}_h\in\mathbb{N}_{h0}\}$. The inf-sup condition follows from Theorem \ref{thm:dsdhcurl}. The proof is completed. 
\end{proof}

\begin{lemma}
Let $(u,\uphi,\uzeta)$ and $(u_h,\uphi{}_h,\uzeta{}_h)$ be the solutions of \eqref{eq:mixbl} and \eqref{eq:mixdisbl}, respectively. Provided sufficient smoothness of $(u,\uphi,\uzeta)$, it holds that 
$$
\|u-u_h\|_{1,\Omega}+\|\uphi-\uphi{}_h\|_{1,\Omega}+\|\uzeta-\uzeta{}_h\|_{\curl,\Omega}\leqslant Ch(\|u\|_{2,\Omega}+\|\uphi\|_{2,\Omega}+\|\uzeta\|_{1,\Omega}+\|\curl\uzeta\|_{1,\Omega}).
$$
\end{lemma}
\begin{proof}
By  Lemma \ref{lem:absconv},  
\begin{multline*}
\|u-u_h\|_{1,\Omega}+\|\uphi-\uphi{}_h\|_{1,\Omega}+\|\uzeta-\uzeta{}_h\|_{\curl,\Omega}
\\
\leqslant C\Big[\inf_{(v_h,\upsi{}_h,\ueta{}_h)\in X_h}(\|u-v_h\|_{1,\Omega}+\|\uphi-\upsi{}_h\|_{1,\Omega}+\|\uzeta-\ueta{}_h\|_{\curl,\Omega}) 
\\
+\sup_{\upsi{}_h\in \undertilde{\mathsf{L}}{}^{+e}_{h0}}\frac{(\upsi-\Pi_h^\mathbb{N}\upsi,\uzeta)+(\curl(\upsi-\Pi_h^\mathbb{N}\upsi),\curl\uzeta)}{\|\upsi{}_h\|_{1,\Omega}}
\\
+\sup_{\ueta{}_h\in\mathbb{N}_{h0}}\frac{(\uphi-\Pi_h^\mathbb{N}\uphi,\ueta{}_h)+(\curl(\uphi-\Pi_h^\mathbb{N}\uphi),\curl\ueta{}_h)}{\|\ueta{}_h\|_{\curl,\Omega}}\Big].
\end{multline*}
The approximation error is controlled in a standard way. By Lemma \ref{lem:intcurlpolyn}, 
\begin{multline*}
(\upsi-\Pi_h^\mathbb{N}\upsi,\uzeta)+(\curl(\upsi-\Pi_h^\mathbb{N}\upsi),\curl\uzeta)=((\upsi-\Pi_h^\mathbb{N}\upsi),\uzeta+\curl\curl\uzeta)
\\
\leqslant \|\upsi-\Pi_h^\mathbb{N}\upsi\|_{0,\Omega}\|\uzeta+\curl\curl\uzeta\|_{0,\Omega}\leqslant Ch|\upsi|_{1,\Omega}(\|\uzeta\|_{1,\Omega}+\|\curl\uzeta\|_{1,\Omega}).
\end{multline*}
By standard error estimate, 
\begin{multline*}
(\uphi-\Pi_h^\mathbb{N}\uphi,\ueta{}_h)+(\curl(\uphi-\Pi_h^\mathbb{N}\uphi),\curl\ueta{}_h)\leqslant \|\uphi-\Pi_h^\mathbb{N}\uphi\|_{\curl,\Omega}\|\ueta{}_h\|_{\curl,\Omega}
\\
\leqslant Ch(\|\uphi\|_{1,\Omega}+\|\curl\uphi\|_{1,\Omega})\|\ueta{}_h\|_{\curl,\Omega}=Ch\|\uphi\|_{1,\Omega}\|\ueta{}_h\|_{\curl,\Omega}.
\end{multline*}
Summing all above completes the proof. 
\end{proof}

\section{A mixed element method for fourth order curl problem}
\label{sec:mixfoc}

In this section, we study the fourth order curl equaltion:
\begin{equation}\label{eq:bvpB}
\left\{
\begin{array}{rl}
\curl^2\mathcal{A}(x)\curl^2\uu+\uu=\uf,&\mbox{in}\ \Omega;
\\
\uu\times\mathbf{n}=\undertilde{0},\ (\curl \uu)\times \mathbf{n}=\undertilde{0} & \mbox{on}\ \partial\Omega,
\end{array}
\right.
\end{equation}
where $\mathcal{A}(x)$ is a symmetric definite bounded matrix fields on $\Omega$.

\subsection{Order reduced formulation}

Its variational problem is (c.f. \cite{Zhang.S2016foc}): given $\uf\in(\undertilde{H}{}^1_0(\curl,\Omega))'$, to find $\uu\in \undertilde{H}{}^1_0(\curl,\Omega)$, such that 
\begin{equation}%\label{eq:primalfoc}
(\mathcal{A}\curl^2\uu,\curl^2 \uv)+(\uu,\uv)=\langle\uf,\uv\rangle,\ \ \forall\,\uv\in \undertilde{H}{}^1_0(\curl,\Omega).
\end{equation}
Note that $\undertilde{H}{}^1_0(\curl,\Omega)$ is configurated by $\{H_0(\curl,\Omega),\undertilde{H}{}^1_0(\Omega),\curl\}$; namely, $\undertilde{H}{}^1_0(\curl,\Omega)=\{\uv\in H_0(\curl,\Omega):\curl\uv\in \undertilde{H}{}^1_0(\Omega)\}$. We rewrite the model problem as: given $\uf{}_1\in H_0(\curl,\Omega)'$ and $\uf{}_2\in\undertilde{H}{}^1_0(\Omega)'$, to find $\uu\in \undertilde{H}{}^1_0(\curl,\Omega)$, such that 
\begin{equation}\label{eq:primalfoc}
(\mathcal{A}\curl^2\uu,\curl^2 \uv)+(\dv\,\curl\uu,\dv\,\curl\uv)+(\uu,\uv)
=\langle\uf{}_1,\uv\rangle +\langle\uf{}_2,\curl\uv\rangle,\ \ \forall\,\uv\in \undertilde{H}{}^1_0(\curl,\Omega).
\end{equation}
Again, a mute term $(\dv\,\curl\uu,\dv\,\curl\uv)$ is added here. 

By Theorem \ref{thm:abs}, a mixed formulation of \eqref{eq:primalfoc} is to find $(\uu,\uphi,\usigma)\in N:=H_0(\curl,\Omega)\times \undertilde{H}{}^1_0(\Omega)\times H_0(\dv,\Omega)$, such that, for any $(\uv,\upsi,\utau)\in N$,
\begin{equation}\label{eq:mixfoc}
\left\{
\begin{array}{ccccl}
(\uu,\uv)&&-(\curl\uv,\usigma)&=&\langle\uf{}_1,\uv\rangle,
\\
&(\mathcal{A}\curl \uphi,\curl\upsi)+(\dv\uphi,\dv\upsi)&(\upsi,\usigma)+(\dv\upsi,\dv\usigma)&=&\langle\uf{}_2,\upsi\rangle,
\\
-(\curl\uu,\utau)&(\uphi,\utau)+(\dv\uphi,\dv\utau)&&=&0.
\end{array}
\right.
\end{equation}
\begin{theorem}
Given $\uf{}_1\in (H_0(\curl,\Omega))'$ and $\uf{}_2\in \undertilde{H}{}^{-1}(\Omega)$, the equation \eqref{eq:mixfoc} admits a unique solution $(\uu,\uphi,\usigma)\in N$, and 
$$
\|\uu\|_{0,\Omega}+\|\curl\uu\|_{1,\Omega}+\|\uphi\|_{1,\Omega}+\|\usigma\|_{\dv,\Omega}\cequiv \|f{}_1\|_{(H_0(\curl,\Omega))'}+\|\uf{}_2\|_{-1,\Omega}.
$$ 
Moreover, $\uu$ solves the primal problem \eqref{eq:primalfoc}.
\end{theorem}
\begin{theorem}
Let $\Omega$ be convex and $\mathcal{A}$ be smooth. Then
\begin{enumerate}
\item if $\uf{}_1\in H(\dv,\Omega)$, then  $\|\uu\|_{1,\Omega}\leqslant C(\|\uf{}_1\|_{\dv,\Omega}+\|\uf{}_2\|_{-1,\Omega})$;
\item if $\uf{}_1,\uf{}_2\in \undertilde{L}^2(\Omega)$, then $\|\usigma\|_{1,\Omega}+\|\dv\usigma\|_{1,\Omega}\leqslant C(\|\uf{}_1\|_{0,\Omega}+\|\uf{}_2\|_{0,\Omega})$;
\item if $\uf{}_1,\uf{}_2\in \undertilde{L}^2(\Omega)$, then  $\|\uphi\|_{2,\Omega}\leqslant C(\|\uf{}_1\|_{0,\Omega}+\|\uf{}_2\|_{0,\Omega})$.
\end{enumerate}
\end{theorem}
\begin{proof}
If $\uf{}_1\in H(\dv,\Omega)$, let $\uv=\nabla q$ for $q\in H^1_0(\Omega)$, then the first equation of \eqref{eq:mixfoc} leads to that $\dv\uu=\dv\uf{}_1\in L^2(\Omega)$, thus $\uu\in H(\dv,\Omega)\cap H_0(\curl,\Omega)\subset \undertilde{H}^1(\Omega)$, and $\|\uu\|_{1,\Omega}\leqslant C(\|\uf{}_1\|_{\dv,\Omega}+\|\uf{}_2\|_{-1,\Omega})$. 

If $\uf{}_1,\uf{}_2\in \undertilde{L}^2(\Omega)$, as $(\curl\uv,\usigma)=(\uu-\uf{}_1,\uv)$ for any $\uv\in H_0(\curl,\Omega)$, we have $\curl\usigma=\uu-\uf{}_1\in\undertilde{L}^2(\Omega)$, and thus $\usigma\in H_0(\dv,\Omega)\cap H(\curl,\Omega)\subset \undertilde{H}^1(\Omega)$, and $\|\usigma\|_{1,\Omega}\leqslant C(\|\usigma\|_{\dv,\Omega}+\|\curl\usigma\|_{0,\Omega})\leqslant C(\|\uf{}_1\|_{0,\Omega}+\|\uf{}_2\|_{0,\Omega})$. Meanwhile, $(\dv\usigma,\dv\upsi)=-(\mathcal{A}\curl\uphi,\curl\upsi)+(\upsi,\usigma)+(\uf{}_2,\upsi)=(\usigma-\curl\mathcal{A}\curl \uphi+\uf{}_2,\upsi)$ for any $\upsi\in \undertilde{H}{}^1_0(\Omega)$, which implies that $\nabla\dv\usigma=\curl\mathcal{A}\curl \uphi-\usigma-\uf{}_2\in\undertilde{L}{}^2(\Omega)$ and further $\dv\usigma\in H^1(\Omega)$, $\|\dv\usigma\|_{1,\Omega}\leqslant C(\|\uf{}_1\|_{0,\Omega}+\|\uf{}_2\|_{0,\Omega})$. Further, by the second equation of \eqref{eq:mixfoc}, $(\mathcal{A}\curl \uphi,\curl\upsi)+(\dv\uphi,\dv\upsi)=(\uf{}_2,\upsi)-(\upsi,\usigma)+(\upsi,\nabla\dv\usigma)$ for any $\uphi\in\undertilde{H}{}^1_0(\Omega)$, thus $\uphi\in \undertilde{H}{}^2(\Omega)\cap \undertilde{H}{}^1_0(\Omega)$, $\|\uphi\|_{2,\Omega}\leqslant C(\|\uf{}_1\|_{0,\Omega}+\|\uf{}_2\|_{0,\Omega})$. The proof is completed.
\end{proof}

\subsection{Discretization}

Define $N_h:=\mathbb{N}_{h0}\times \undertilde{\mathsf{L}}{}^{+f}_{h0}\times\mathbb{RT}_{h0}$. A discretization scheme is to find $(\uu{}_h,\uphi{}_h,\usigma{}_h)\in N_h$, such that, for $(\uv{}_h,\upsi{}_h,\utau{}_h)\in N_h$,
\begin{equation}\label{eq:mixdisfoc}
\left\{
\begin{array}{ccccl}
(\uu{}_h,\uv{}_h)&&(-\curl\uv{}_h,\usigma{}_h)&=&\langle\uf{}_1,\uv{}_h\rangle,
\\
&(\mathcal{A}\curl \uphi{}_h,\curl\upsi{}_h)+(\dv\uphi{}_h,\dv\upsi{}_h)&(\Pi_h^\mathbb{RT}\upsi{}_h,\usigma{}_h)+(\dv\Pi_h^\mathbb{RT}\upsi{}_h,\dv\usigma{}_h)&=&\langle\uf{}_2,\uv{}_h\rangle,
\\
-(\curl\uu{}_h,\utau{}_h)&(\Pi_h^\mathbb{RT}\uphi{}_h,\utau{}_h)+(\dv\Pi_h^\mathbb{RT}\uphi{}_h,\dv\utau{}_h)&&=&0.
\end{array}
\right.
\end{equation}

\begin{lemma}
Given $\uf{}_1\in (H_0(\curl,\Omega))'$ and $\uf{}_2\in \undertilde{H}{}^{-1}(\Omega)$, the problem \eqref{eq:mixdisfoc} admits a unique solution $(\uu{}_h,\uphi{}_h,\usigma{}_h)\in Y_h$, and 
$$
\|\uu{}_h\|_{\curl,\Omega}+\|\uphi{}_h\|_{1,\Omega}+\|\usigma{}_h\|_{\dv,\Omega}\cequiv \|\uf{}_1\|_{\curl',h}+\|\uf{}_2\|_{-1,h},
$$
where
$$
\|\uf{}_1\|_{\curl',h}:=\sup_{\uv{}_h\in \mathbb{N}_{h0}\setminus\{\mathbf{0}\}}\frac{\langle\uf{}_1,\uv{}_h\rangle}{\|\uv{}_h\|_{\curl,\Omega}},\ \ \mbox{and}\ \ \|\uf{}_2\|_{-1,h}:=\sup_{\upsi{}_h\in\undertilde{\mathsf L}{}_{h0}^{+f}\setminus\{\mathbf{0}\}}\frac{\langle\uf{}_2,\upsi{}_h\rangle}{\|\upsi{}_h\|_{1,\Omega}}.
$$
\end{lemma}
\begin{proof}
Again, we verify the hypotheses of  Lemma \ref{lem:intcurlpolyn}. By the error estimate of $\Pi_h^\mathbb{RT}$, (c.f., e.g., Propositions 2.5.1, 2.5.3 and 2.5.4 of \cite{Boffi.D;Brezzi.F;Fortin.M2013}) $\|\Pi_h^\mathbb{RT}\upsi{}_h\|_{0,\Omega}\leqslant C(\|\upsi{}_h\|_{0,\Omega}+h\|\upsi{}_h\|_{1,\Omega})\leqslant C\|\upsi{}_h\|_{0,\Omega}$ for $\upsi{}_h\in\undertilde{\mathsf L}{}_{h0}^{+f}$. The boundedness of the bilinear forms involved follows then. Besides, the coercivity of $(u_h,u_h)+(\mathcal{A}\curl \uphi{}_h,\curl\uphi{}_h)+(\dv\uphi{}_h,\dv\uphi{}_h)$ holds on $Z_h':=\{(u_h,\uphi{}_h)\in \mathbb{N}_{h0}\times \undertilde{\mathsf L}{}_{h0}^{+f}:-(\curl\uu{}_h,\utau{}_h)+(\Pi_h^\mathbb{RT}\uphi{}_h,\utau{}_h)+(\dv\Pi_h^\mathbb{RT}\uphi{}_h,\dv\utau{}_h)=0,\ \forall\,\utau{}_h\in\mathbb{RT}_{h0}\}.$ Finally, the inf-sup condition needed holds by Theorem \ref{thm:dsdhdiv}. The proof is completed. 
\end{proof}

\begin{lemma}
Let $(\uu,\uphi,\usigma)$ and $(\uu{}_h,\uphi{}_h,\usigma{}_h)$ be the solutions of \eqref{eq:mixfoc} and \eqref{eq:mixdisfoc}, respectively. Provided sufficient smoothness of $(\uu,\uphi,\usigma)$, it holds that
\begin{multline}
\|\uu-\uu{}_h\|_{\curl,\Omega}+\|\uphi-\uphi{}_h\|_{1,\Omega}+\|\usigma-\usigma{}_h\|_{\dv,\Omega}
\\
\leqslant Ch(\|\uu\|_{1,\Omega}+\|\curl\uu\|_{1,\Omega}+\|\uphi\|_{2,\Omega}+\|\usigma\|_{1,\Omega}+\|\dv\usigma\|_{1,\Omega}).
\end{multline}
\end{lemma}

\begin{proof}
Again, by Lemma \ref{lem:absconv},
\begin{multline}
\|\uu-\uu{}_h\|_{\curl,\Omega}+\|\uphi-\uphi{}_h\|_{1,\Omega}+\|\usigma-\usigma{}_h\|_{\dv,\Omega}
\\
\leqslant C{\Big [}{}\inf_{(\uv{}_h,\upsi{}_h,\utau{}_h)\in Y_h}(\|\uu-\uv{}_h\|_{\curl,\Omega}+\|\uphi-\upsi{}_h\|_{1,\Omega}+\|\usigma-\utau{}_h\|_{\dv,\Omega})
\\
+\sup_{\upsi{}_h\in \undertilde{\mathsf{L}}{}^{+f}_{h0}\setminus\{\mathbf{0}\}}\frac{(\upsi{}_h-\Pi_h^{\mathbb{RT}}\upsi{}_h,\usigma)+(\dv\upsi{}_h-\dv\Pi_h^{\mathbb{RT}}\upsi{}_h,\dv\usigma)}{\|\upsi{}_h\|_{1,\Omega}}
\\
+\sup_{\utau{}_h\in\mathbb{RT}_{h0}}\frac{(\uphi-\Pi_h^\mathbb{RT}\uphi,\utau{}_h)+(\dv\uphi-\dv\Pi_h^\mathbb{RT}\uphi,\dv\utau{}_h)}{\|\utau{}_h\|_{\dv,\Omega}}\Big].
\end{multline}
The estimation of the approximation error is immediate. 

By standard error estimate of $\Pi_h^\mathbb{RT}$, $\|\Pi_h^\mathbb{RT}\upsi{}_h-\upsi{}_h\|_{0,\Omega}\leqslant Ch\|\upsi{}_h\|_{1,\Omega}$. Let $\mathcal{P}_h$ be the $L^2$ projection operator to the space of piecewise constants, then 
\begin{multline*}
(\dv\Pi_h^\mathbb{RT}\upsi{}_h-\dv\upsi{}_h,\dv\usigma)_K=(\mathcal{P}_h\dv\upsi{}_h-\dv\upsi{}_h,\dv\usigma)_K
\\
=(\mathcal{P}_h\dv\upsi{}_h-\dv\upsi{}_h,\dv\usigma-c)_K\leqslant Ch^2\|\dv\upsi{}_h\|_{1,K}\|\dv\usigma\|_{1,K}\leqslant Ch\|\upsi{}_h\|_{1,K}\|\dv\usigma\|_{1,K}.
\end{multline*} 
Thus $(\upsi{}_h-\Pi_h^{\mathbb{RT}}\upsi{}_h,\usigma)+(\dv\upsi{}_h-\dv\Pi_h^{\mathbb{RT}}\upsi{}_h,\dv\usigma)\leqslant Ch\|\upsi{}_h\|_{1,\Omega}(\|\usigma\|_{0,\Omega}+\|\dv\usigma\|_{1,\Omega})$.

Meanwhile, $\|\Pi_h^\mathbb{RT}\uphi-\uphi\|_{0,\Omega}\leqslant Ch\|\uphi\|_{1,\Omega}$, and $\dv\Pi_h^\mathbb{RT}\uphi-\dv\uphi=0-0=0.$ Thus 
$$
(\uphi-\Pi_h^\mathbb{RT}\uphi,\utau{}_h)+(\dv\uphi-\dv\Pi_h^\mathbb{RT}\uphi,\dv\utau{}_h)\leqslant Ch\|\uphi\|_{1,\Omega}\|\utau{}_h\|_{0,\Omega}.
$$

Summing all above completes the proof.
\end{proof}

\section{Concluding remarks}
\label{sec:con}

In this paper, we study the general methodology on the construction of order reduced schemes for fourth order problems. Under the framework presented in Section \ref{sec:abs}, a high-regularity problem can be transformed to an equivalent system on three low-regularity spaces. Once a ``configurability" condition is verified for a certain fourth order problem, the existence of the three spaces is guaranteed, and its equivalent system is given constructively. As second-order Sobolev spaces can usually be treated to consist of functions in first-order spaces whose derivatives are in first order spaces, we may expect a wide applicability of the framework for fourth order problems. This will be discussed further in future.

A stable decomposition $Y=BR+S$ (c.f. Section \ref{sec:abs}) is a basic tool for the framework. Its discretised analogue is fundamental for construction of numerical schemes. The regularity of $Y$ can usually be lower than $S$, and thus higher-degree polynomials are usually needed for discretising $S$ than for discretising $Y$. This way, it might be not easy to keep discretised $S$ contained in discretised $Y$; this is a key difficulty for constructing discretised stable decompositions, especially ones with low-degree polynomials. In Section \ref{sec:rds}, we overcome this obstacle by the aid of interpolations and construct two discretised regular decompositions, which are probably among the ones of lowest degree. Order reduced numerical schemes are thus constructed for two fourth order problems. With respect to the framework of order reduction, different stable decompositions, continuous and discrete, will be motivated, constructed and applied for various fourth order problems.

As the order of the problem is reduced, flexibility can be expected for the numerical schemes. For example, finite elements of lower order will be more probable to be contained in finite element programming packages, and convenience can be brought in on implementations. Also, low order finite element spaces can be nested, thus nested discretisation schemes are admitted. This can be an advantage for designing multilevel methods. We refer to \cite{Zhang.S;Xi.Y;Ji.X2016} for a preliminary example. Other utilisations will be discussed in future. 

The framework for problems with more general formulations will be discussed in future, such as problems which are not necessarily self-adjoint or on Hilbert spaces, problems with parameters dependence, and etc.. A framework for problems of orders higher than four will also be discussed.

\section*{Acknowledgement}
%\thanks{
The author is supported by the National Natural Science Foundation of China with Grant No. 11471026 and National Centre for Mathematics and Interdisciplinary Sciences, Chinese Academy of Sciences.%}

\end{document}